\documentclass[12pt]{amsart}

\usepackage{amsmath,amssymb}
\usepackage{amsfonts}
\usepackage{amsthm}
\usepackage{bm}
\usepackage{mathtools}
\usepackage{enumitem}

\newtheorem{theorem}{Theorem}[section]
\newtheorem{lemma}[theorem]{Lemma}
\newtheorem{proposition}[theorem]{Proposition}
\newtheorem{definition}[theorem]{Definition}
\newtheorem{corollary}[theorem]{Corollary}

\newtheorem{prop}{Proposition}[section]
\newtheorem{remark}[prop]{Remark}

\def\*#1{\mathbf{#1}}
\def\gperp{{\nabla^\perp}}
\def\R{{\tilde{R}}}
\def\vp{{\varphi}}
\def\pR{{{}^\perp \!R}}
\def\tr{{\mathrm{tr}\,}}
\def\tmperp{{\left(TM\right)^\perp}}

\def\He{\mathrm {Hess}\,}

\subjclass[2000]{Primary 53C42, 53C40, 52B25. }

\date{\today}

\begin{document}
\raggedbottom

\title{Parallel Codazzi tensors with submanifold applications}

\author{Anthony Gruber}


\begin{abstract} 
A decomposition theorem is established for a class of closed Riemannian submanifolds immersed in a space form of constant sectional curvature.  In particular, it is shown that if $M$ has nonnegative sectional curvature and admits a Codazzi tensor with ``parallel mean curvature'', then $M$ is locally isometric to a direct product of irreducible factors determined by the spectrum of that tensor. This decomposition is global when $M$ is simply connected, and generalizes what is known for immersed submanifolds with parallel mean curvature vector.

\vspace{1pc}

\end{abstract}

\maketitle
\tableofcontents

\section{Introduction}
A Codazzi tensor on the Riemannian manifold $M$ is a symmetric (0,2)-tensor field $\vp$ whose covariant derivative $\nabla\vp$ is totally symmetric along $M$ (c.f. section 2). Though defined algebraically, this notion is motivated by a fundamental compatibility condition from the geometry of immersed submanifolds.  Indeed, when $M$ is a hypersurface immersed in a locally symmetric space, the classical Codazzi equation reduces to precisely the requirement that the second fundamental form $h$ be a Codazzi tensor.

It is well-known that careful study of the second fundamental form is essential for determining both intrinsic and extrinsic properties of submanifold immersions (see e.g. \cite{chern1968,simons1968,yano1971,yau1975,chen2010}), so it is natural to suspect that other Codazzi tensors may also be useful for this purpose.  Consequently, there has been much investigation into the structure and properties of Codazzi tensors in the abstract which often complements the study of immersions, e.g. \cite{ferus1981,simon1981,bourguignon1981,derdzinski1983,catino2015,shandra2019} and the references therein.  In particular, it is known that any Codazzi tensor $\vp$ on a complete Riemannian manifold $M$ commutes with the Ricci tensor \cite{bourguignon1981}, and that if $M$ has constant sectional curvature $K$ then $\vp$ admits the local expression $\vp_f = \He f + Kf\,\mathrm{Id}_M$ for some smooth function $f: M \to \mathbb{R}$  \cite{ferus1981}.  Moreover, when $M$ is closed with nonnegative sectional curvature, it has recently been shown that every trace-free Codazzi tensor is invariant under parallel translations \cite{shandra2019} \textemdash which is well-known when the tensor in question is the second fundamental form of a minimal immersion \cite{yano1971}. 

Despite their high geometric relevance, the structure of manifolds which admit Codazzi tensors is comparatively less understood.  In particular, given a manifold $M$ which admits a Codazzi tensor, almost nothing is known about the global structure of $M$ regardless of hypotheses.  However, some local results have been demonstrated. For example, when $M$ is complete, the leaves of the eigenspace distribution generated by any Codazzi tensor are locally integrable and umbilic in $M$ \cite{derdzinski1981}. Moreover, whenever $M$ has dimension at least three and admits a non-parallel Codazzi tensor with constant trace and exactly two distinct eigenvalues, $M$ is locally isometric to a warped product \cite{derdzinski1981}.

The contribution of this work is to establish a global decomposition for a certain class of closed and simply connected submanifolds $M^n$ immersed in a space form $\mathbb{M}^{n+p}(c)$ of constant sectional curvature $c\in\mathbb{R}$ and admitting a Codazzi tensor with respect to $\mathbb{M}(c)$.  In particular, if $M$ has everywhere nonnegative sectional curvature and satisfies a parallelism condition on its ``mean curvature vector'' (c.f. section 2), then the following is shown.

\begin{theorem}\label{thm:main}
Let $M^n \subset \mathbb{M}^{n+p}(c)$ be a closed submanifold with nonnegative sectional curvature, and let $\vp$ be a Codazzi tensor on $M$.  If the ``mean curvature vector'' $\bm{\Phi} = \phi \,\*e_{n+1}$ is parallel in $\tmperp$, then
\[|\nabla\vp^{n+1}| = 0, \qquad  \sum R^i_{jij}\left(\lambda_i^{n+1} - \lambda_j^{n+1}\right)^2 = 0.\]
Moreover, $M$ is locally isometric to the product $M = M_1 \times M_2 \times ... \times M_l$ where $l$ is the number of distinct eigenvalues of $\vp^{n+1}$ and $TM_i$ is spanned by those eigenvectors which have eigenvalue $\lambda_i^{n+1}$.  If $M$ is simply connected, then this statement is global.
\end{theorem}

This result provides local-to-global information about the structure of submanifolds admitting Codazzi tensors, and confirms that submanifolds with parallel mean curvature (with respect to any Codazzi tensor) are somewhat rigid\textemdash a fact which is known already for the second fundamental form \cite{yano1971, yau1975}.  The sequel is devoted to the proof of this result, as well as the extraction of a few  Corollaries which elucidate the structure of manifolds admitting Codazzi tensors in the presence of additional information.

\vspace{1pc}

\textbf{Acknowledgements:}  The author would like to acknowledge Prof. Magdalena Toda, whose attention and advice during the preparation of this manuscript was invaluable to its completion.

\section{Preliminaries}
Some basics of submanifold geometry are now recalled in the formalism of Chern \cite{chern1968}.  Let $M^n$ be a closed (i.e. compact without boundary)  $n$-dimensional Riemannian manifold isometrically immersed in the $(n+p)$-dimensional space form $\mathbb{M}^{n+p}(c)$ of constant sectional curvature $c\in\mathbb{R}$.  Around any point $x\in M$, choose a local orthonormal basis $\*e_1,...,\*e_n,\*e_{n+1},...,\*e_{n+p}$ for $T\mathbb{M}(c)$ such that $\*e_1,...,\*e_n$ form a local orthonormal basis for $TM$.  Adopting the index conventions
\begin{equation*}
    1\leq A,B,C \leq n+p, \qquad 1 \leq i,j,k \leq n, \qquad n+1\leq \alpha,\beta,\gamma \leq n+p,
\end{equation*}
it follows that there are dual one-forms $\omega^A$ defined by $\omega^A(\*e_B) = \delta^A_B$ which form a (local) basis for the cotangent bundle $T^*\mathbb{M}(c)$. The structure equations for the natural Levi-Civita connection on $\mathbb{M}(c)$ are then expressed as
\begin{align*}
    d\omega^A + \sum_B \omega^A_B \wedge \omega^B &= 0,  \\
    d\omega^A_B + \sum_C\omega^A_C\wedge\omega^C_B &= \frac{1}{2}\sum_{C,D}\R^A_{BCD}\,\omega^C\wedge\omega^D,
\end{align*}
where $d$ is the exterior derivative and $\R$ is the Riemann curvature tensor on $\mathbb{M}(c)$.


It is straightforward to verify (see e.g. \cite{kobayashi1963}) that the immersion of $M$ in $\mathbb{M}(c)$ induces connections on $TM$ and $TM^\perp$ given by appropriate restrictions of the above.  Denoting the Riemann curvature tensors of these connections by $R$ and $\pR$, respectively, there are the venerable submanifold equations of Gauss, Codazzi, and Ricci, stated as the following lemma.
\begin{lemma}\label{lem:stdeqns}
Let $h\in T^*M \otimes T^*M \otimes \left(TM\right)^\perp$ denote the second fundamental form of $M \subset \mathbb{M}(c)$.  In the notation above,
\begin{align}
    R^i_{jkl} &= \tilde{R}^i_{jkl} + \sum_\alpha h_{ik}^\alpha h_{jl}^\alpha - h_{il}^\alpha h_{jk}^\alpha, \qquad \mathrm{(Gauss\, equation)} \label{eq:gauss} \\
    h_{ij,k}^\alpha &= h_{ik,j}^\alpha + \tilde{R}^\alpha_{ikj}, \qquad\qquad\qquad\qquad \mathrm{(Codazzi\,equation)} \label{eq:codazzi} \\
    \pR^\alpha_{\beta jk} &= \tilde{R}^{\alpha}_{\beta jk} + \sum_i h_{ij}^\alpha h_{ik}^\beta - h_{ik}^\alpha h_{ij}^\beta. \qquad \mathrm{(Ricci \, equation)} \label{eq:ricci}
\end{align}
\end{lemma}
\begin{proof}
A calculation is provided, for convenience, in the Appendix to this manuscript.
\end{proof}

Recall that a powerful technique for studying the geometry of submanifolds involves computation of the metric Laplace operator.  In particular, once commutation relations are established for derivatives along the submanifold, self-adjointness of the Laplacian can be used along with knowledge of the curvature tensor in order to conclude results about submanifold structure.  Due originally to Bochner \cite{bochner1969} and later popularized by Simons \cite{simons1968} and Chern \cite{chern1968}, this technique has been widely utilized for studying the behavior of closed hypersurface immersions \cite{chern1968, cheng1977, chen2010}.  Fortunately, similar analysis can also prove useful for studying more general submanifolds which admit Codazzi tensors.

To that end, note that the ambient manifold $\mathbb{M}(c)$ has constant sectional curvature $c$, so that the tensor $\R$ admits the simplified expression (c.f \cite[Chapter 5]{kobayashi1963})
\begin{equation}\label{eq:Rspaceform}
\R^A_{BCD} = \left(\delta^A_C\,\delta_{BD} - \delta^A_D\,\delta_{BC}\right)c.
\end{equation}
Using this, it follows immediately from the Codazzi equation (\ref{eq:codazzi}) that the covariant derivatives of the second fundamental form satisfy  $h_{ij,k}^\alpha = h_{ik,j}^\alpha$ on $M$  and hence $\nabla h$ is totally symmetric.  This motivates the following standard definition.

\begin{definition}
A tensor field $\vp \in T^*M \otimes T^*M \otimes \left(TM\right)^\perp$ which is symmetric in its lower indices is said to satisfy the Codazzi equation with respect to $\mathbb{M}(c)$ provided
\[\vp^\alpha_{ij,k} =\vp^\alpha_{ik,j},\]
for all indices $i,j,k,\alpha$.  In this case, $\vp$ is said to be a Codazzi tensor on the submanifold $M \subset \mathbb{M}(c)$.
\end{definition}

Moreover, there is the following commutation result for derivatives which is valid for any symmetric tensor field on $M$.

\begin{lemma}\label{lem:commutator}
Let $\vp \in T^*M \otimes T^*M \otimes \left(TM\right)^\perp $ be symmetric in its lower indices.  Then, the Hessian of $\varphi$ obeys the following commutation rule,
\begin{equation*}
    \vp^\alpha_{ij,kl} - \vp^\alpha_{ij,lk} = \sum_{\beta,m}\left(\vp^\beta_{ij}\pR^\alpha_{\beta lk} - \vp^\alpha_{mj}R^m_{ilk} - \vp^\alpha_{im}R^m_{jlk}\right).
\end{equation*}
\end{lemma}
\begin{proof}
The necessary computation can be found in the Appendix.
\end{proof}

To apply the aforementioned technique of computing the Laplace operator to the present situation, it must be extended to submanifold immersions with higher codimension. Here, it is advantageous to assume some control on the normal bundle $\tmperp$.  There are several reasonable assumptions that can be made to this end, but the two most natural for the current purposes are:
\begin{enumerate}
    \item That $\tmperp$ is locally parallelizable (i.e. flat),
    \item That the mean curvature vector is parallel in $\tmperp$.
\end{enumerate}
Enforced either separately or together, these conditions have led to interesting results which often generalize their classical codimension-one counterparts (e.g. \cite{zheng1997, yano1971, yau1975, chen2010}). Of course, when $M$ is a hypersurface, $\tmperp$ is flat automatically, and condition (2) is equivalent to $M$ having constant mean curvature (c.f. the discussion below). Both conditions (1) and (2) will be examined in the sequel, however most emphasis will be placed on (2).


To adapt the notion of parallel mean curvature vector for compatibility with more general Codazzi tensors, let $\tr{\vp^\alpha} = \sum_i \vp^\alpha_{ii}$ and recall that when $\vp = h$ is the second fundamental form, the mean curvature vector of $M$ is given by 
\[\*H = \frac{1}{n}\sum\left(\tr{h^\alpha}\right)\*e_\alpha.\]
When studying the geometry of submanifolds, it is useful to adapt the local frame on $\tmperp$ by choosing $\*e_{n+1}$ in the direction of $\*H$ (c.f. \cite{yau1974, zheng1997, chen2010}).  Similarly, it is gainful for the present purpose to choose $\*e_{n+1}$ in the direction of the generalized ``mean curvature vector'',
\[\*\Phi = \frac{1}{n} \sum \left(\tr{\vp^\alpha}\right)\*e_\alpha.\]
With this adaptation, $\*\Phi = \phi\,\*e_{n+1}$ where $\phi = (1/n)\tr{\vp^{n+1}}$ and $\tr{\vp^\alpha} = 0$ for all $\alpha \neq n+1$.  Note that such a choice is always possible away from points where $\*\Phi = \*0$, while if $\*\Phi = \*0$ then $\tr{\vp^\alpha}=0$ for all $\alpha$ in any basis $\{\*e_\alpha\}$ for $\tmperp$.  Moreover, recall that a tensor field is said to be parallel in a vector bundle if its covariant derivative vanishes along sections of this bundle.  It follows that $\*\Phi$ is parallel in $\tmperp$ provided
\[0 = \gperp \*\Phi = d\phi\,\*e_{n+1} + \phi\gperp \*e_{n+1} = d\phi\,\*e_{n+1} + \phi\sum_\alpha \omega^\alpha_{n+1}\*e_\alpha,\]
which implies that $\phi$ is constant and (when $\phi\neq 0$) the connection forms $\omega^{n+1}_\alpha = 0$ for all $\alpha$. Hence, $\*e_{n+1}$ is parallel in $\tmperp$ whenever $\*\Phi \neq \*0$ is also.  Consequently, there is the following definition.

\begin{definition}\label{def:parallelMC}
Let $\vp\in T^*M \otimes T^*M \otimes \tmperp$ be a Codazzi tensor and let $\*\Phi = \sum \left(\tr{\vp^\alpha}\right)\*e_\alpha = \phi \,\*e_{n+1}$ be its mean curvature vector.  Then, $\vp$ is said to have parallel mean curvature provided $d\phi = 0$ and $\omega^{n+1}_\beta = 0$ for all $\beta$.
\end{definition}

\begin{remark}
When $\varphi = h$ is the second fundamental form, $\*\Phi = \*H$ is the mean curvature vector of $M$ and this is precisely the concept of parallel mean curvature which generalizes the CMC condition of hypersurfaces \cite{chen2010}.
\end{remark}

Now, recall that it follows from the Ricci equation (\ref{eq:ricci}) and (\ref{eq:Rspaceform}) that the shape operators $\vp^\alpha = \sum_{i,j}\vp_{ij}^\alpha\, \omega^i \otimes \omega^j$ may be diagonalized simultaneously if and only if the tensor $\pR$ vanishes, i.e. if and only if the normal bundle $\tmperp$ is flat. In this case, there is an orthogonal transformation of $TM$ such that  $\vp^\alpha_{ij} = \lambda_i^\alpha \delta_{ij}$ for all $i,j,\alpha$.  On the other hand, if $\vp$ has parallel mean curvature (but $\tmperp$ is not necessarily flat), then $\*e_{n+1}$ is parallel and the structure equations show that $\pR^{n+1}_{\alpha ij} = 0$ for all $\alpha$.  In this case, there is a basis for $TM$ where $\varphi^{n+1}_{ij} = \lambda^{n+1}_i \delta_{ij}$.  With this, it is now appropriate to compute the metric Laplacian in the case of either assumption (1) or (2) above. 



\begin{lemma}\label{lem:simons}
Suppose that $\vp$ satisfies Codazzi equations relative to the space form $\mathbb{M}^{n+p}(c)$, and let $\*\Phi  =\phi\,\*e_{n+1}$.  If $\vp$ has parallel mean curvature (c.f. Definition~\ref{def:parallelMC}), then
\begin{equation}\label{eq:MCVsimons}
    \frac{1}{2}\Delta|\vp^{n+1}|^2 = |\nabla\vp^{n+1}|^2 + \frac{1}{2}\sum_{i,j} R^i_{jij}\left(\lambda^{n+1}_i - \lambda^{n+1}_j\right)^2.
\end{equation}
On the other hand, suppose $\tmperp$ is flat, but $\*\Phi$ is not necessarily parallel.  Then, it follows that
\begin{equation}\label{eq:flatsimons}
    \frac{1}{2}\Delta|\vp|^2 = |\nabla\vp|^2 + \sum \lambda^\alpha_i \left(\tr{\vp^\alpha}\right)_{ii} + \frac{1}{2}\sum  R^i_{jij}\left(\lambda^\alpha_i - \lambda^\alpha_j\right)^2.
\end{equation}
\end{lemma}

\begin{proof}
Using Lemma~\ref{lem:commutator}, the Laplacian $\Delta\vp$ can be expressed componentwise as
\begin{equation*}
\begin{split}
    \Delta \vp_{ij}^\alpha &= \sum_k \vp^\alpha_{ij,kk} = \sum_k \left(\vp^\alpha_{ij,kk}-\vp^\alpha_{ik,jk}\right) \\
    &\quad+ \sum_k \left(\vp^\alpha_{ik,jk} - \vp^\alpha_{ik,kj}\right) + \sum_k \left(\vp^\alpha_{ik,kj} - \vp^\alpha_{kk,ij}\right) + \sum_k\vp^\alpha_{kk,ij} \\
    &= \sum_k \left(\vp^\alpha_{ij,kk} - \vp^\alpha_{ik,jk}\right) + \sum_k \left(\vp^\alpha_{ik,kj} - \vp^\alpha_{kk,ij}\right) + \left(\tr{\vp^\alpha}\right)_{,ij} \\
    &\quad+ \sum_{\beta,k,m} \left(\vp^\beta_{ik}\pR^\alpha_{\beta kj} - \vp^\alpha_{km}R^m_{ikj} - \vp^\alpha_{im}R^m_{kkj} \right).
\end{split}
\end{equation*}
Using the fact that $\vp$ satisfies Codazzi equations, this expression reduces further to
\begin{equation*}
    \Delta\vp^\alpha_{ij} = \left(\tr{\vp}^\alpha\right)_{,ij} + \sum_{\beta,k,m} \left(\vp^\beta_{ik}\pR^\alpha_{\beta kj} - \vp^\alpha_{km}R^m_{ikj} - \vp^\alpha_{im}R^m_{kkj} \right).
\end{equation*}
Now, if $\tmperp$ is flat, $\pR\equiv 0$ and we may write $\vp^\alpha_{ij} = \lambda_i^\alpha \delta_{ij}$ for all $\alpha,i,j$.  In this case, 
\begin{equation*}
\begin{split}
    \frac{1}{2}\Delta|\vp|^2 &= |\nabla\vp|^2 + \sum \vp^\alpha_{ij}\Delta\vp^\alpha_{ij} \\
    &= |\nabla\vp|^2 + \sum \vp_{ij}^\alpha\left(\tr{\vp^\alpha} \right)_{,ij} - \sum \left(\vp_{ij}^\alpha\vp^\alpha_{km}R^m_{ikj} + \vp_{ij}^\alpha\vp^\alpha_{im}R^m_{kkj}\right) \\
    &= |\nabla\vp|^2 + \sum \lambda_i^\alpha\left(\tr{\vp^\alpha} \right)_{,ii} - \sum\left( \lambda_i^\alpha\lambda_j^\alpha R^i_{jij} + \left(\lambda_i^\alpha\right)^2 R^i_{jji}\right) \\
    &= |\nabla\vp|^2 + \sum \lambda_i^\alpha\left(\tr{\vp^\alpha} \right)_{,ii} + \frac{1}{2}\sum R^i_{jij}\left(\lambda_i^\alpha - \lambda_j^\alpha\right)^2,
\end{split}
\end{equation*}
where the last line follows from skew-symmetrization in $i,j$.  On the other hand, if $\*\Phi$ is parallel in the (not necessarily flat) bundle $\tmperp$, it follows that there is an orthogonal transformation of $TM$ such that $\varphi^{n+1}_{ij} = \lambda^{n+1}_i \delta_{ij}$, and a completely similar calculation yields the desired expression.
\end{proof}

Equations (\ref{eq:MCVsimons}) and (\ref{eq:flatsimons}) represent an application of the Bochner/Simons/Chern technique mentioned previously, and are a generalization of relationships known to hold in the case of the second fundamental form \cite{yau1975,zheng1997}.  To immediately illustrate the utility of this approach, it is possible to extract the following generalization of \cite[Corollary 2.8]{simon1981}.  
\begin{proposition}
Suppose $M$ is closed with positive sectional curvature, and $\tmperp$ is flat.  If $\vp,\psi$ are Codazzi tensors on $M$ with $\tr{\vp^\alpha} = \tr{\psi^\alpha}$ for all $\alpha$, then $\vp = \psi$.
\end{proposition}

\begin{proof}
Consider $\xi = \vp - \psi$.  This tensor is obviously Codazzi and satisfies $\tr{\xi^\alpha} = 0$ for all $\alpha$, so assuming a frame which diagonalizes $\xi$,  Lemma~\ref{lem:simons} shows that
\[\frac{1}{2}\Delta|\xi|^2 = |\nabla\xi|^2 + \frac{1}{2}\sum R^i_{jij}\left(\lambda^\alpha_i - \lambda^\alpha_j\right)^2.\]
Since $M$ is closed, Stokes' theorem implies 
\[0 = \int|\nabla \xi|^2 + \int R^i_{jij}\left(\lambda^\alpha_i - \lambda^\alpha_j\right)^2,\]
so that all eigenvalues of $\xi$ are constant and $\lambda_i^\alpha = \lambda_j^\alpha = c^\alpha$ for all $\alpha,i,j$.  Since, each $\xi^\alpha$ is trace-free, it must hold that each $c^\alpha = 0$.  Hence, $\xi^\alpha \equiv 0$, so that $\vp^\alpha = \psi^\alpha$ for all $\alpha$, yielding the conclusion.
\end{proof}

\section{Proof and corollaries}
The goal of this section is to prove Theorem~\ref{thm:main} and extract some Corollaries which give insight into the strucure of submanifolds admitting Codazzi tensors in the presence of additional conditions.  First, recall the following Lemma from \cite{derdzinski1981}, which conveys spectral information about Codazzi tensors in the general case. 

\begin{lemma}[Derdzinski]\label{lem:andrzej}
Suppose $A = A^i_j\, \omega^j \otimes \*e_i$ is a Codazzi tensor in (1,1)-form on the Riemannian manifold $M$ and $\*u,\*v$ are eigenvectors of $A$ with eigenvalue $\lambda$.  Then, 
\[ A\left(\nabla_\*v\*u\right) = \lambda \nabla_\*v\*u + d\lambda(\*v)\*u - \langle \*u,\*v \rangle \nabla \lambda. \]
In particular, if $\lambda$ is constant, then the submanifold $M_\lambda \subset M$ tangent to the eigenvector distribution $V_\lambda$ is totally geodesic.
\end{lemma}

\begin{proof}
Since this result is used in the argument of Theorem~\ref{thm:main}, a proof is provided in the Appendix.
\end{proof}

The proof of Theorem~\ref{thm:main} will now be given.  The argument is thematically inspired by \cite[Theorem 9]{yau1974}, and uses the holonomy-invariance of the eigenspaces of $\vp^{n+1}$ in order to construct an appropriate De Rham decomposition of $M$.  Though this construction is generically local, simple-connectedness removes its dependence on homotopy and a global conclusion is made possible (c.f. \cite[Chapter 4.6]{kobayashi1963}).

\begin{proof}[Proof of Theorem~\ref{thm:main}]
Since $\*\Phi$ is parallel in $\tmperp$, it follows from Stokes' theorem and Lemma~\ref{lem:simons} that
\[0 = \int_M |\nabla\vp^{n+1}|^2 + \sum R^i_{jij}\left(\lambda_i^{n+1} - \lambda_j^{n+1}\right)^2 = 0,\]
so that the hypotheses imply the first statement.
It follows that each $\lambda_i^{n+1}$ is constant, and the structure equations imply that 
\[0 = d\varphi^{n+1}_{ij} = \left(\lambda_i^{n+1} - \lambda_j^{n+1}\right) \omega^i_j, \]
for all $i,j$.  Denoting $\lambda_k^{n+1}:= \lambda_k$, it follows that $\omega^i_j = 0$ when $\lambda_i \neq \lambda_j$.  Ordering the eigenvalues so that \[\lambda_1 = \lambda_2 = ... = \lambda_{n_1} > \lambda_{n_1 + 1} = ... = \lambda_{n_2} > ... > \lambda_{n_{k-1} + 1} = ... = \lambda_{n_k} > ... = \lambda_n,\]
the above discussion implies that for $n_k < i \leq n_{k+1}$,
\[ d\omega^i = -\sum_{j=n_k + 1}^{n_{k+1}} \omega^i_j \wedge \omega^j.\]
Therefore, as $k$ varies, the distributions of $TM$ defined by $\{\omega^{n_k + 1} = .... = \omega^{n_{k+1}} = 0\}$ are closed under exterior differentiation, hence integrate to mutually orthogonal foliations of $M$.  Moreover, since each $\lambda_i$ is constant, Lemma~\ref{lem:andrzej} shows that the eigenspaces corresponding to distinct eigenvalues are invariant under parallel translation and therefore totally geodesic in $M$. Since $M$ is geodesically complete when compact, De Rham's decomposition theorem implies that the maximal integral manifolds of each eigenspace distribution define a product decomposition $M = M_1 \times M_2 \times ... \times M_l$ where $1\leq l\leq n$ is the number of distinct eigenvalues $\lambda_i$ (c.f. \cite[Chapter 4]{kobayashi1963}). By the same theorem, this decomposition is global when $M$ is  simply connected.
\end{proof}


\begin{remark}
When $\vp = h$, $\*\Phi$ is the mean curvature vector $\*H$, we partially recover \cite[Theorem 9]{yau1975}.  Indeed, in this case it can be further demonstrated that $M = M_1 \times M_2 \times ... \times M_k$ such that each $M_i$ is minimal in a totally umbilical $N_i$ of positive codimension in $\mathbb{M}(c)$.
\end{remark}

The remainder is devoted to corollaries.  First, recall that a manifold $M$ is said to be \textit{indecomposable} provided its tangent space contains no proper and nondegenerate subspace which is invariant under parallel translation. 

\begin{corollary}\label{cor:indecomp}
Suppose $M \subset \mathbb{M}(c)$ is an immersed submanifold with positive sectional curvature which admits a Codazzi tensor $\vp$ with parallel mean curvature.  Then, $\vp^{n+1}$ is a constant multiple of the metric tensor and $M$ is indecomposable.
\end{corollary}

\begin{proof}
Since $M$ has positive sectional curvature, Theorem~\ref{thm:main} shows that all eigenvalues of $\vp^{n+1}$ are constant and identical, say equal to $\lambda$, and that $TM$ is spanned by the eigenvectors of $\vp^{n+1}$ with eigenvalue $\lambda$.  Therefore, $\vp^{n+1}_{ij} = \lambda\,\delta_{ij}$ and the first conclusion follows.  Now, if there were a proper nondegenerate subspace of $TM$ left invariant by parallel translation, say $TM_c$, then the constancy of $\lambda$ would imply that $TM_c$ integrates to a totally geodesic submanifold of $M$ so that locally $M = M_c \times (M\setminus M_c)$.  Let $\{\omega^I\}$ resp. $\{\omega^i\}$ be the dual forms to $M$ resp. $M_c$.  Then, if $K$ is an index such that $\omega^K = 0$ on $M_c$, total geodesy implies that $\omega^K_i = \sum_j h_{ij}^K\,\omega^j = 0$, hence the structure equations show that the sectional curvatures $R^K_{ijl}$ must vanish on $M_c$, a contradiction.
\end{proof}

In addition to indecomposability, there is the following consequence when $M$ is a hypersurface with strictly positive sectional curvature. 

\begin{corollary}\label{cor:positive}
Let $M^n \subset \mathbb{M}^{n+1}(c)$ be a closed hypersurface with positive sectional curvature.  Then, $M$ is indecomposable and any Codazzi tensor $\varphi = \varphi_{ij}\,\omega^i \otimes \omega^j$ which has constant trace is a constant multiple of the metric.
\end{corollary}

\begin{proof}
Since $M$ is a hypersurface, its normal bundle is trivial and the vector $\*\Phi = \sum \vp_{ii}\,\*e_{n+1}$ is parallel if and only if $\vp$ has constant trace.  By Theorem~\ref{thm:main}, all eigenvalues of $\varphi$ are constant and equal.  Therefore, there is a constant function $c:M \to \mathbb{R}$ such that $\varphi_{ij} = c \,\delta_{ij}$ and the second conclusion follows.  Indecomposibility follows immediately from Corollary~\ref{cor:indecomp}.
\end{proof}

\begin{remark}
In the case that the immersion is trivially represented by $\mathrm{Id}_M: M \to M$, $\*\Phi$ is automatically parallel in $\tmperp = \varnothing$ and the argument in Corollary~\ref{cor:positive} recovers the well-known fact that any Codazzi tensor with constant trace on a closed Riemannian manifold with positive sectional curvature must be a multiple of the metric (c.f. \cite[Theorem 16.9]{besse2007}).
\end{remark}

\begin{remark}
When $\vp = h$, then Corollary~\ref{cor:positive} applies to any closed CMC hypersurface $M$ with positive sectional curvature.  Moreover, the conclusion recovers the fact that $M$ must be totally umbilical and isometric to the standard sphere (c.f. \cite{yano1971,cheng1977,chen2010}).
\end{remark}

Finally, recall that a manifold $M$ is said to have \textit{harmonic curvature} when its Ricci tensor $\mathrm{Ric}_{ij} = \sum_k R^k_{ikj}$ satisfies the Codazzi equation.  Such manifolds arise in the study of Yang-Mills theory (c.f. \cite{besse2007}), and are intimately related to the \textit{Einstein manifolds} which have diagonal Ricci tensor.  Using the work above, this connection is highlighted in the following corollary.

\begin{corollary}
Let $M$ be a closed, simply connected hypersurface with parallel Ricci tensor and nonnegative sectional curvature.  Then, $M= M_1 \times M_2 \times ... \times M_l$ such that each $M_i$ is spanned by the eigenvectors of $\mathrm{Ric}$ with eigenvalue $\lambda_i$.  Moreover, if the sectional curvature of $M$ is strictly positive, then all eigenvalues of $\mathrm{Ric}$ are the same and $M$ is both indecomposable and Einstein.
\end{corollary}

\begin{proof}
First, notice that the assumption on the Ricci tensor guarantees that $0 = \nabla_k R_{ij} = \nabla_i R_{jk} = \nabla_j R_{ik}$, so $\mathrm{Ric}$ is a Codazzi tensor and $M$ has constant scalar curvature. The rest now follows from Theorem~\ref{thm:main}, Corollary~\ref{cor:indecomp}, and Corollary~\ref{cor:positive}.
\end{proof}

\bibliographystyle{abbrv}

\providecommand{\bysame}{\leavevmode\hbox to3em{\hrulefill}\thinspace}
\providecommand{\MR}{\relax\ifhmode\unskip\space\fi MR }
\providecommand{\MRhref}[2]{%
  \href{http://www.ams.org/mathscinet-getitem?mr=#1}{#2}
}
\providecommand{\href}[2]{#2}

\section*{Appendix}
For convenience, the calculations necessary for the results obtained in the body are presented here.  Note that the Einstein summation convention is employed throughout, so that indices appearing both up and down in a tensor expression are implicitly summed over their appropriate range. 

\begin{proof}[Proof of Lemma~\ref{lem:stdeqns}]
To obtain the well-known Gauss, Codazzi, and Ricci equations, it suffices to compare the structure equations on $M$ with the pullbacks of those on $\mathbb{M}(c)$.  More precisely, it follows from the structure equations of $M$ that
\begin{equation*}
    d\omega^i_j = -\omega^i_k\wedge\omega^k_j + \frac{1}{2}R^i_{jkl}\,\omega^k\wedge\omega^l.
\end{equation*}
On the other hand, the structure equations on $\mathbb{M}(c)$ pull back to yield
\begin{equation*}
\begin{split}
    \*r^*(d\omega^i_j) &= -\omega^i_K\wedge\omega^K_j + \frac{1}{2}\tilde{R}^i_{jKL}\,\omega^K\wedge\omega^L \\ 
    &= -\omega^i_k\wedge\omega^k_j - \omega^i_\alpha\wedge\omega^\alpha_j + \frac{1}{2}\tilde{R}^i_{jkl}\,\omega^k\wedge\omega^l \\
    &= -\omega^i_k \wedge \omega^k_j + h_{ik}^\alpha h_{jl}^\alpha\,\omega^k \wedge \omega^l + \frac{1}{2}\tilde{R}^i_{jkl}\,\omega^k\wedge\omega^l,
\end{split}
\end{equation*}
since $\omega^\alpha = 0$ and $\omega^\alpha_i = h_{ij}^\alpha\,\omega^j$ on $M$.
Comparing these representations, it follows that
\begin{equation*}
    0 = \*r^*(d\omega^i_j) - d\omega^i_j = \left(h_{ik}^\alpha h_{jl}^\alpha + \frac{1}{2}\tilde{R}^i_{jkl} - \frac{1}{2}R^i_{jkl}\right)\omega^k\wedge\omega^l.
\end{equation*}
Now, since $\omega^k\wedge\omega^l$ is skew-symmetric in $k,l$, the coefficient functions must be also.  Skew-symmetrizing this expression and noting the symmetries of $R, \tilde{R}$ then yields
\begin{equation*}
    0 = h_{ik}^\alpha h_{jl}^\alpha - h_{il}^\alpha h_{jk}^\alpha + \tilde{R}^i_{jkl} - R^i_{jkl},
\end{equation*}
which immediately implies the Gauss equation
\begin{equation*}
    R^i_{jkl} = \tilde{R}^i_{jkl} + \sum_\alpha h_{ik}^\alpha h_{jl}^\alpha - h_{il}^\alpha h_{jk}^\alpha.
\end{equation*}
This procedure is now repeated to yield the other equations in question.  More precisely, notice that pulled back to $M$, there is
\begin{equation*}
\begin{split}
    \*r^*(d\omega^\alpha_i) &= -\omega^\alpha_J \wedge\omega^J_i + \frac{1}{2}\tilde{R}^\alpha_{iJK}\,\omega^J \wedge \omega^K \\
    &= -\omega^\alpha_j\wedge\omega^j_i - \omega^\alpha_\beta\wedge\omega^\beta_i + \frac{1}{2}\tilde{R}^\alpha_{ijk}\,\omega^j\wedge\omega^k \\
    &= h_{jk}^\alpha\,\omega^j_i\wedge\omega^k - h_{ik}^\beta\,\omega^\alpha_\beta\wedge\omega^k + \frac{1}{2}\tilde{R}^\alpha_{ijk}\,\omega^j\wedge\omega^k.
\end{split}
\end{equation*}
Recalling the definition of covariant derivative, it follows that
\begin{equation*}
    h_{ij,k}^\alpha\,\omega^k = dh_{ij}^\alpha - h_{kj}^\alpha\,\omega^k_i - h_{ik}^\alpha\,\omega^k_j + h_{ij}^\beta\,\omega^\alpha_\beta,
\end{equation*}
and so differentiation directly on $M$ also yields
\begin{equation*}
\begin{split}
    d\omega_i^\alpha &= d\left(h_{ij}^\alpha\,\omega^j\right) = dh_{ij}^\alpha\wedge\omega^j - h_{ij}^\alpha\, \omega^j_k \wedge \omega^k \\
    &= \left(h_{ij,k}^\alpha\,\omega^k + h_{kj}^\alpha\,\omega^k_i + h_{ik}^\alpha\,\omega^k_j - h_{ij}^\beta\,\omega^\alpha_\beta \right)\wedge\omega^j - h_{ij}^\alpha\,\omega^j_k\wedge\omega^k \\
    &= \left(h_{ij,k}^\alpha\,\omega^k + h_{kj}^\alpha\,\omega^k_i - h_{ij}^\beta\,\omega^\alpha_\beta \right)\wedge\omega^j.
\end{split}
\end{equation*}
Using this, direct comparison again shows
\begin{equation*}
    0 = \*r^*(d\omega^\alpha_i) - d\omega^\alpha_i = \left(h_{ij,k}^\alpha + \frac{1}{2}\tilde{R}^\alpha_{ijk}\right)\omega^j\wedge\omega^k,
\end{equation*}
from which skew-symmetrization returns
\begin{equation*}
    h_{ij,k}^\alpha - h_{ik,j}^\alpha + \tilde{R}^\alpha_{ijk} = 0.
\end{equation*}
This establishes the Codazzi equation of the submanifold $M$,
\begin{equation*}
    h_{ij,k}^\alpha = h_{ik,j}^\alpha + \tilde{R}^\alpha_{ikj}.
\end{equation*}
For the remainder, note that on $M$
\begin{equation*}
\begin{split}
    \*r^*(d\omega^\alpha_\beta) &= -\omega^\alpha_i\wedge\omega^i_\beta - \omega^\alpha_\gamma\wedge\omega^\gamma_\beta + \frac{1}{2}\tilde{R}^\alpha_{\beta jk}\,\omega^j\wedge\omega^k \\
    &= -\omega^\alpha_\gamma\wedge\omega^\gamma_\beta + \left(h_{ij}^\alpha h_{ik}^\beta + \frac{1}{2}\tilde{R}^\alpha_{\beta jk}\right)\omega^j\wedge\omega^k.
\end{split}
\end{equation*}
Moreover, the structure equations in the normal bundle to $M$ imply
\begin{equation*}
    d\omega^\alpha_\beta = -\omega^\alpha_\gamma\wedge\omega^\gamma_\beta + \frac{1}{2}\pR^\alpha_{\beta jk}\,\omega^j\wedge\omega^k.
\end{equation*}
Therefore, it follows that
\begin{equation*}
    0 = \*r^*(d\omega^\alpha_\beta) - d\omega^\alpha_\beta = \left(h_{ij}^\alpha h_{ik}^\beta + \frac{1}{2}\tilde{R}^\alpha_{\beta jk} - \frac{1}{2}\pR^\alpha_{\beta jk}\right)\omega^j\wedge\omega^k,
\end{equation*}
and skew-symmetrization in $j,k$ yields
\begin{equation*}
    h_{ij}^\alpha h_{ik}^\beta - h_{ik}^\alpha h_{ij}^\beta + \tilde{R}^{\alpha}_{\beta jk} - \pR^\alpha_{\beta jk} = 0,
\end{equation*}
which is equivalent to the Ricci equation,
\begin{equation*}
    \pR^\alpha_{\beta jk} = \tilde{R}^{\alpha}_{\beta jk} + \sum_i h_{ij}^\alpha h_{ik}^\beta - h_{ik}^\alpha h_{ij}^\beta.
\end{equation*}
This completes the calculation.
\end{proof}

\begin{proof}[Proof of Lemma~\ref{lem:commutator}]
The commutation relationship between second derivatives on $M$ is now computed.  First, notice that the definition of exterior derivative implies
\begin{equation*}
\begin{split}
    0 &= d\left(d\vp^\alpha_{ij}\right) = d\vp_{ij,k}^\alpha\wedge\omega^k + \vp_{ij,k}^\alpha\,d\omega^k + d\vp^\alpha_{kj}\wedge\omega^k_i + \vp^\alpha_{kj}\,d\omega^k_i + d\vp^\alpha_{ik}\wedge\omega^k_j \\
    &\quad+ \vp^\alpha_{ik}\,d\omega^k_j - d\vp^\beta_{ij}\wedge\omega^\alpha_\beta - \vp^\beta_{ij}\,d\omega^\alpha_\beta \\
    &= d\vp_{ij,k}^\alpha\wedge\omega^k - \vp_{ij,k}^\alpha\,\omega^k_l\wedge\omega^l + d\vp^\alpha_{kj}\wedge\omega^k_i - \vp^\alpha_{kj}\,\omega^k_l\wedge\omega^l_i - \vp^\alpha_{kj}\,\omega^k_\beta\wedge\omega^\beta_i \\
    &\quad+ \frac{1}{2}\vp_{kj}^\alpha \R_{ilm}^k\,\omega^l\wedge\omega^m + d\vp_{ik}^\alpha \wedge\omega^k_j - \vp_{ik}^\alpha\,\omega^k_l\wedge\omega^l_j - \vp^\alpha_{ik}\,\omega^k_\beta\wedge\omega^\beta_j + \frac{1}{2}\vp_{ik}^\alpha\R^k_{jlm}\,\omega^l\wedge\omega^m \\
    &\quad- d\vp_{ij}^\beta\wedge\omega^\alpha_\beta + \vp_{ij}^\beta\,\omega^\alpha_k\wedge\omega^k_\beta + \vp_{ij}^\beta\,\omega^\alpha_\gamma\wedge\omega^\gamma_\beta - \frac{1}{2}\vp_{ij}^\beta\R^\alpha_{\beta lm}\,\omega^l\wedge\omega^m.
\end{split}
\end{equation*}
Using the fact that $\omega_i^\alpha = h_{ij}^\alpha\,\omega^j$ on $M$ and relabeling indices when necessary, the above computation continues as
\begin{equation*}
\begin{split}
    0 &= \left(d\vp_{ij,k}^\alpha - \vp_{ij,l}^\alpha\,\omega^l_k - \vp_{jl,k}^\alpha\,\omega^l_i - \vp_{il,k}^\alpha\,\omega^l_j + \vp_{ij,k}^\beta\,\omega^\alpha_\beta \right)\wedge\omega^k \\
    &\quad+ \left(\vp_{jm}^\alpha h_{ml}^\beta h_{ik}^\beta + \vp_{im}^\alpha h_{ml}^\beta h_{jk}^\beta - \vp_{ij}^\beta h_{ml}^\alpha h_{mk}^\beta \right)\omega^l\wedge\omega^k \\
    &\quad+ \frac{1}{2}\left( \vp_{mj}^\alpha\R^m_{ilk} + \vp^\alpha_{im}\R^m_{jlk} - \vp^\beta_{ij}\R^\alpha_{\beta lk} \right)\omega^l\wedge\omega^k \\
    &= \bigg(\vp_{ij,kl}^\alpha + \vp_{jm}^\alpha h_{ml}^\beta h_{ik}^\beta + \vp_{im}^\alpha h_{ml}^\beta h_{jk}^\beta - \vp_{ij}^\beta h_{ml}^\alpha h_{mk}^\beta \\
    &\quad+ \frac{1}{2}\vp_{mj}^\alpha\R^m_{ilk} + \frac{1}{2}\vp^\alpha_{im}\R^m_{jlk} - \frac{1}{2}\vp^\beta_{ij}\R^\alpha_{\beta lk} \bigg)\omega^l\wedge\omega^k.
\end{split}
\end{equation*}
The desired expression now follows from skew-symmetrization in $l,k$ and applying the equations of Gauss and Ricci.  More precisely,
\begin{equation*}
\begin{split}
    0 &= \vp_{ij,kl}^\alpha - \vp_{ij,lk}^\alpha + \vp_{mj}^\alpha \left(h_{ml}^\beta h_{ik}^\beta - h_{mk}^\beta h_{il}^\beta \right) + \vp^\alpha_{im}\left(h_{ml}^\beta h_{jk}^\beta - h_{mk}^\beta h_{jl}^\beta \right) \\
    &\quad+ \vp_{ij}^\beta\left(h_{mk}^\beta h_{ml}^\alpha - h_{ml}^\beta h_{mk}^\alpha\right) + \vp^\alpha_{jm}\R^m_{ilk} + \vp_{im}^\alpha\R^m_{jlk} - \vp^\beta_{ij}\R^\alpha_{\beta lk} \\
    &= \vp_{ij,kl}^\alpha - \vp_{ij,lk}^\alpha + \vp_{mj}^\alpha\left(R^m_{ilk} - \R^m_{ilk}\right) + \vp_{im}^\alpha\left(R^m_{jlk} - \R^m_{jlk}\right) - \vp_{ij}^\beta\left(\pR^\alpha_{\beta lk} - \tilde{R}^\alpha_{\beta lk}\right) \\
    &\quad+ \vp^\alpha_{jm}\R^m_{ilk} + \vp_{im}^\alpha\R^m_{jlk} - \vp^\beta_{ij}\R^\alpha_{\beta lk} \\
    &= \vp^\alpha_{ij,kl} - \vp^\alpha_{ij,lk} - \vp^\beta_{ij}\pR^\alpha_{\beta lk} + \vp^\alpha_{mj}R^m_{ilk} + \vp^\alpha_{im}R^m_{jlk},
\end{split}
\end{equation*}
which was to be shown.  
\end{proof}

\begin{proof}[Proof of Lemma~\ref{lem:andrzej}]
First, the structure equations imply  \[\nabla_\*v\*u = d\*u(\*v) = \nabla u^i(\*v)\*e_i,\] where $\nabla u^i = u^i_j + u^j\omega_j^i$.  So, \[ A\left(\nabla_\*v\*u\right) = A^i_j\, \omega^j\left( \nabla u^k(\*v)\*e_k \right) \*e_i = A^i_j \nabla u^j(\*v)\*e_i. \]
Now, by the Leibniz rule 
\begin{align*}
     A_j^i \nabla u^j(\*v) &= \nabla\left( u^j A_j^i \right)(\*v) - u^j\nabla A^i_j(\*v) \\
     &= \nabla\left(\lambda u^j\delta_{ij} \right)(\*v) - A^i_{j,k} u^j v^k = \lambda \nabla u^i(\*v) + u^i d\lambda(\*v) - A^i_{j,k} u^j v^k. 
\end{align*}
Moreover, using the Codazzi equation and the fact that $\*u,\*v$ are eigenvectors,
\[A^i_{j,k} u^j v^k = A^j_{k,i} u^j v^k = \lambda_{,i} \delta^j_{k}u^j v^k = \langle \*u,\*v \rangle \lambda_{,i}.\]
The desired computational formula now follows.  To finish the proof, note that if $\lambda$ is constant, then $A\left(\nabla_\*v\*u\right) = \nabla_{A(\*v)}\*u = \lambda \nabla_\*v\*u$.  So, if $\nabla_\*u\*u = 0,$ then so does $A\left(\nabla_\*u\*u\right)$ and the distribution $V_\lambda$ is closed under parallel transport.  Hence, $V_\lambda$ is totally geodesic.
\end{proof}

\end{document}